\documentclass[11pt,reqno]{amsart}
\usepackage{amsmath,amssymb,amsthm,amsfonts,enumerate,hyperref, caption, verbatim}
\theoremstyle{plain}
\newtheorem{theorem}{Theorem}[section]
\newtheorem{theorem2}{Theorem}[section]
\newtheorem{proposition}[theorem2]{Proposition}
\newtheorem{cor}[theorem]{Corollary}
\newtheorem{lemma}[theorem2]{Lemma}

\theoremstyle{definition}
\newtheorem{definition}[theorem2]{Definition}

\theoremstyle{remark}

\numberwithin{equation}{section}

\newcommand{\Riem}{\mathrm{Rm}}

\newcommand{\vol}{\mathrm{Vol}}

\newcommand{\Ric}{\mathrm{Ric}}

\newcommand{\tr}{\mathrm{tr}}

\renewcommand{\div}{\mathrm{div}}

\begin{document}
\title{Destabilising compact warped product Einstein manifolds}
\begin{abstract}
The linear stability of warped product Einstein metrics as fixed points of the Ricci flow is investigated. We generalise the results of Gibbons, Hartnoll and Pope and show that in sufficiently low dimensions, all warped product Einstein metrics are unstable.  By exploiting the relationship between warped product Einstein metrics, quasi-Einstein metrics and Ricci solitons, we introduce a new destabilising perturbation (the Ricci variation) and show that certain infinite families of warped product Einstein metrics will be unstable in high dimensions. 
\end{abstract}
\author{Wafaa Batat}
\address{Ecole Nationale Polytechnique d'Oran,
	B.P 1523 El M'naouar, 31000 Oran, Algeria}
\email{batatwafa@yahoo.fr}
\author{Stuart Hall}
\address{School of Mathematics and Statistics, Herschel Building, Newcastle University, Newcastle-upon-Tyne, NE1 7RU} 
\email{stuart.hall@ncl.ac.uk}
\author{Thomas Murphy}
\address{Department of Mathematics, California State University Fullerton, 800 N. State College Bld., Fullerton, CA 92831, USA.}
\email{tmurphy@fullerton.edu}
\maketitle

\section{Introduction}
\label{sec:1}
\subsection{Main results}
In 2003, Perelman made spectacular use of Hamilton's Ricci flow to prove Thurston's geometrization conjecture \cite{Per1}, \cite{Per3} and \cite{Per2}. Put simply, geometrization says that three-dimensional manifolds decompose into pieces that can each be endowed with a canonical geometry. In order to extend geometrization to higher dimensions, a crucial step is finding the right set of canonical geometries in each dimension.  One candidate for these geometries are metrics that are the \textit{stable} fixed points (up to diffeomorphism and scaling) of the Ricci flow. Fixed points of the Ricci flow are known as \textit{Ricci solitons} and clearly include Einstein metrics. Roughly speaking, stability can be taken to mean that the Ricci flow starting at small perturbation of a Ricci soliton will return to the soliton.\\
\\
The study of the linear stability of Ricci solitons was initiated by Cao, Hamilton, and Ilmanen \cite{CHI} who considered the second variation of Perelman's $\nu$ entropy, a monotonic quantity for the flow. In the subsequent years, stability questions for many important classes of metric have been investigated such as: Einstein metrics admitting parallel spinors \cite{DWW}, compact symmetric spaces \cite{CH}, and K\"ahler metrics \cite{HMPAMS}.\\
\\
In this article we take up the study of stability for a class of compact Einstein metrics known as warped products. These are generalisations of ordinary Riemannian products where the underlying manifold $M$ decomposes as $M=B\times F$ for a base $B$ and a fibre $F$ but the metric on the fibre is `warped' by a factor coming from the base. We refer the reader to Section \ref{sec:2} for the precise definition of these metrics. There are many important examples of such Einstein metrics in low dimensions including the inhomogenous families on $\mathbb{S}^{3}\times \mathbb{S}^{2}, \ \mathbb{S}^{3}\times \mathbb{S}^{3}$,  and $\mathbb{S}^{4}\times \mathbb{S}^{2}$ due to B\"ohm \cite{Bo}, and the warped product Einstein metric on $\mathbb{CP}^{2}\sharp\overline{\mathbb{CP}}^{2}\times\mathbb{S}^{2}$ due to L\"u, Page and Pope \cite{LPP}.   We completely settle the question of stability in low dimensions:

\begin{theorem}\label{thmA}
Let $(M^{n},g)$ be a warped product Einstein manifold where $n\leq 6$. Then  $(M,g)$ is unstable as a fixed point of the Ricci flow.  	
\end{theorem}
We remark that the proof of Theorem \ref{thmA} also shows that warped products with three-dimensional base and a four-dimensional fibre are unstable too. However, this does not account for all possible seven-dimensional products.\\
\\ 
Many examples of Einstein warped products are constructed from  families of Riemannian manifolds $(B,\bar{g}_{i},f_{i},m_{i})$ where the metrics $\bar{g}_{i}$ (known as \textit{quasi-Einstein metrics}) converge in the $C^{\infty}$ topology as $i\rightarrow \infty$ to a Ricci soliton $(B,\bar{g}_{\infty},\phi)$.  For example, Case has demonstrated that the metrics of L\"u--Page--Pope have this property \cite{Case2}. For warped products coming from such a family, we prove an asymptotic instability result that shows that the Einstein metrics are also unstable if the dimension is large enough.
\begin{theorem}\label{thmB}
	Let $(B,\bar{g}_{i},f_{i},m_{i})$ be a sequence of quasi-Einstein metrics that converge in the $C^{\infty}$ topology to a non-trivial Ricci soliton $(B,g_{\infty},\phi)$.\\
	Then there exists a $K\in \mathbb{N}$ such that the associated warped product Einstein metrics $(M^{k},g)$ are unstable for $k\geq K$. 
\end{theorem}
As discussed in Section \ref{sec:2}, the construction of a warped product Einstein metric requires an Einstein metric $\tilde{g}$ on the fibre $F$. The following theorem makes precise the interaction between the stability of the fibre Einstein metric and that of the warped product.
\begin{theorem}\label{thmC}
	Let $(M,g)$ be a warped product Einstein manifold with fibre $(F,\tilde{g})$ satisfying ${\Ric(\tilde{g})=\mu \tilde{g}}$. Let $\sigma$ be a divergence-free, trace-free eigentensor of the Lichnerowicz Laplacian of $(F,\tilde{g})$ satisfying ${\widetilde{\Delta}_{L}\sigma=-\kappa\sigma}$. If  $\kappa < \mu$ then the warped product $(M,g)$ is unstable.
\end{theorem}
This theorem allows us to find a large class of Ricci flow unstable warped products. We will refer to warped products that are unstable in the manner of Theorem \ref{thmC} as fibre-unstable.
\begin{cor}\label{Cor1}
	The following fibres yield fibre-unstable warped products:
	\begin{itemize}
		\item when $(F,\tilde{g})$ is a Riemannian product $(F_{1}\times F_{2}, g_{1}\oplus g_{2})$, 
		\item when $(F,\tilde{g})$ is a K\"ahler-Einstein metric with $h^{(1,1)}>1$,
		\item when $(F,\tilde{g})$ is a fibre-unstable warped product Einstein metric.
	\end{itemize}	
\end{cor}
\subsection{The stability of generalised black holes}
A second motivation for studying stability, which historically preceeds the Ricci flow, comes from the role compact Einstein metrics play in the theory of generalised black holes. This was the context of the pioneering study of the stability of B\"ohm's Einstein metrics on low-dimensional products of spheres conducted by Gibbons, Hartnoll, and Pope \cite{GHP}. In their study they exploited the fact that the B\"ohm metrics are invariant under a cohomogeneity one action by a compact Lie group. We generalise their results in low dimensions to arbitrary warped product Einstein metrics without any symmetry assumptions.
\begin{theorem}\label{thmE}
Let $M$ be a warped product Einstein metric with three-dimensional base and two or three-dimensional fibre.  Then the associated  Schwarzschild--Tangherlini black hole is unstable. 
\end{theorem}
\subsection{Relation of Theorem \ref{thmE} to other works on stability}
The construction of Schwarzschild--Tangherlini black holes from Einstein metrics with positive Einstein constant can itself be viewed as a non-compact warped product.  There are many other methods that build geometrically interesting manifolds from positive Einstein metrics. For example, in \cite{HHS} a stability inequality for Ricci-flat cones over Einstein bases was investigated.  In this setting, as in Theorem \ref{thmE}, the stability of the cone is related to the spectral properties of the Lichnerowicz Laplacian of the base. The authors proved that the cones over product manifolds in dimension less than 10 are unstable (which is somewhat similar to the situation in Theorem \ref{thmE} in that the construction over low dimensional `products' is unstable).  In \cite{HHS} authors also discussed the conjectural link, due to Ilmanen, between unstable cones and the non-uniqueness of Ricci flows emerging from them. Kr\"oncke has made a more extensive study of such non-compact warped product constructions in \cite{KlKr3} and \cite{KlKr4}.\\
\\
The relation between the linear stability discussed in this article and the dynamic stability of the Ricci flow has been made precise by Sesum in \cite{Sesum}.  We refer the reader to Section \ref{LSfRF} for more details on this relationship.\\

\subsection{Conventions}
As the proofs of the theorems rest on certain quantities having a particular sign, it is important to state exactly the conventions used in the paper where there is room for ambiguity. On the product ${B\times F}$, uppercase letters denote general coordinates, lowercase Roman letters denote the coordinates on $B$ and lowercase Greek letters denote the coordinates on the fibre $F$. We will denote the associated coordinate vector fields by $\partial_{A}$, $\partial_{a}$ and $\partial_{\alpha}$. In order to try and keep coordinate calculations uncluttered, we will often abuse notation by writing $\nabla_{A}T$ for $\nabla_{\partial_{A}}T$ and $\nabla_{A}B$ for $\nabla_{\partial_{A}}\partial_{B}$.\\ 
\\
The divergence of a tensor $T$ is given by
$$\div(T)(\cdot) = g^{AB}(\nabla_{A}T)(\partial_{B},\cdot)$$
and the (connection) Laplacian $\Delta T$ is given by
$$\Delta T = -\nabla^{\ast}\nabla T = g^{AB}(\nabla^{2}_{A,B}T)=\div(\nabla T).$$
With this convention, the spectrum of the Laplacian is non-positive. The convention we use for curvature is 
$$R(X,Y,Z,W) = g(R(X,Y)Z,W) = g(\nabla^{2}_{Y,X}Z-\nabla^{2}_{X,Y}Z,W),$$
and the curvature operator $\Riem: s^{2}(TM)\rightarrow  s^{2}(TM)$ is given by
$$\Riem(h,\cdot)_{AB} = R_{ACBD}h^{CD},$$
for $h \in s^{2}(TM)$.\\
\\
Geometric objects on the base manifold $B$ will usually be denoted using a bar, for example $\bar{g}$ for the metric.  Geometric objects on the fibre $F$ are likewise denoted using a tilde, so we have $\tilde{g}$ for the fibre metric. We will set $n=\dim(B)$ and $m=\dim(F)$.\\
\\ 
\emph{Acknowledgements:} The initial work on this project was conducted whilst WB and TM were visiting SH in Dec 2015 and Jan 2016.  The visit of WB was funded by an LMS grant under scheme 4. TM was funded by a Cal. State Fullerton Startup Grant.   We would like to thank the referees for their careful reading and very useful comments.  
\section{Background} 
\label{sec:2}
\subsection{Warped Product Einstein metrics}
Let $M=B\times F$ be a product manifold. Equip $M$ with the metric
\begin{equation*}
g = \pi^{\ast}_{B}\overline{g} \oplus (f\circ\pi_{B})^{2}\pi^{\ast}_{F}\tilde{g},
\end{equation*}
where $\overline{g}$ and $\tilde{g}$ are Riemannian metrics on $B$ and $F$ respectively, ${\pi_{B}:M\rightarrow B}$ and ${\pi_{F}:M\rightarrow F}$ denote the natural projections, and $f\in C^{\infty}(B)$. The Riemannian manifold $(M,g)$ is referred to as a \emph{warped product}. We shall henceforth adopt the standard abuse of notation and 
drop the references to the projections $\pi_{B}$ and $\pi_{F}$. The manifold $B$ is referred to as the base and $F$ is referred to as the fibre. By taking the function $f$ to be constant, one recovers a Riemannian product.\\
\\
We will be concerned with the case where the warped product is an Einstein metric with positive scalar curvature. By Myers's theorem this immediately implies the manifolds $B$ and $F$ are compact. If a warped product $(M,g)$ is an Einstein manifold with ${\Ric(g)=\lambda g}$ for $\lambda>0$, then the following is well known (e.g. Corollary 9.107 in \cite{Bes}) :
\begin{equation*}
(F,\tilde{g}) \textrm{ is an Einstein manifold with } \Ric(\tilde{g})=\mu \tilde{g} \textrm{ for some } \mu>0,
\end{equation*}
\begin{equation} \label{KKeq}
f\overline{\Delta} f+(m-1)|\overline{\nabla}f|^{2}+\lambda f^{2} =\mu,
\end{equation}
\begin{equation} \label{QEMeq} 
\Ric(\overline{g})-mf^{-1}\overline{\nabla}^{2}f = \lambda \overline{g}.
\end{equation}

Riemannian manifolds $(B,\overline{g})$ that solve Equation (\ref{QEMeq}) for some ${f \in C^{\infty}(B)}$ and $m>0$ are known as \emph{quasi-Einstein} manifolds and are studied in their own right as interesting generalisations of Einstein metrics. A foundational result of Kim and Kim \cite{KK} says that if $(B,\overline{g},f,m)$ solve (\ref{QEMeq}), then there exists a $\mu>0$ such that $f$ solves Equation (\ref{KKeq}). Hence for integral values of $m \geq 2$, one can construct warped product Einstein manifolds from a quasi-Einstein metric on the base $B$.  \\
\\
We also consider Riemannian manifolds $(M,g)$ where the metric $g$ solves
\begin{equation}\label{GRSeq}
\Ric(g)+\nabla^{2}\phi=\lambda g,
\end{equation}
where $\phi \in C^{\infty}(M)$ and $\lambda \in \mathbb{R}$.  Metrics solving (\ref{GRSeq}) are called \emph{gradient Ricci solitons}. One can view Equation (\ref{GRSeq}) as the formal limit as $m\rightarrow \infty$ of (\ref{QEMeq}) by setting $\phi_{m}=-m\log f$. More detailed results on the sense in which gradient Ricci solitons are the limits of quasi-Einstein metrics can be found in the work of Case \cite{Case}.\\  
\\
At the time of writing, there are very few general methods for constructing of compact, warped product Einstein metrics. Much more is known about constructions of non-compact warped product Einstein metrics, see for example Chapter 9 of \cite{Bes} where there are examples of quasi-Einstein metrics on non-compact surfaces with the parameter $m\in (1,\infty)$.  As mentioned in Section \ref{sec:1}, the first compact examples that were found are due to B\"ohm \cite{Bo}  and occur on the product ${\mathbb{S}^{n}\times F^{m}}$ with $2\leq m\leq 6$ and ${3 \leq n \leq 9-m}$. The second family of examples come from a construction due to L\"u, Page and Pope \cite{LPP} of quasi-Einstein metrics on $\mathbb{CP}^{1}$-bundles over a Fano K\"ahler--Einstein base. This construction (and its generalisation due to the second author \cite{Halljgp}) produce quasi-Einstein metrics for all $m>1$ and hence infinitely many examples of warped product Einstein manifolds.  The lowest dimensional examples of the L\"u--Page--Pope construction occur when the base is the non-trivial $\mathbb{CP}^{1}$-bundle over $\mathbb{CP}^{1}$.  In this case one can view the base as ${B=\mathbb{CP}^{2}\sharp\overline{\mathbb{CP}}^{2}}$. A very explicit construction of the L\"u--Page--Pope metrics on this manifold was given in \cite{BHJM}. As mentioned already, Case \cite{Case2} showed that the L\"u--Page--Pope quasi-Einstein metrics converge as $m\rightarrow \infty$ to the Koiso--Cao K\"ahler--Ricci soliton constructed independently  in \cite{Cao} and \cite{Koi}. The metrics constructed in \cite{Halljgp} should converge to generalisations of the Koiso--Cao soliton due to Dancer and Wang \cite{DWC1} which are known as Dancer--Wang K\"ahler--Ricci solitons. 

\subsection{Linear stability for Ricci flow} \label{LSfRF}
Ricci solitons first arose as the fixed points up to gauge of the Ricci flow
$$\frac{\partial g}{\partial t} = -2\Ric(g).$$
In particular, Einstein metrics evolve via homothetic rescaling. Perelman \cite{Per1} introduced a quantity $\nu$ which is monotonically increasing along a Ricci flow and stationary only at shrinking Ricci solitons and, in particular, at Einstein metrics with positive Einstein constant $\lambda$. If the second variation of $\nu$ at an Einstein metric is positive, then a small perturbation of the metric will increase $\nu$ and the Ricci flow cannot flow back to the Einstein metric. Hence the Einstein metric will be unstable. The linear stability of Ricci solitons was considered by Cao, Hamilton and Ilmanen \cite{CHI},  Cao and Zhu \cite{CZ}, and the second and third authors \cite{HMPAMS}.  In order to state their theorem we need to introduce the operators 
$$\div_{\phi}(\cdot) := e^{-\phi}\div(e^{\phi}\cdot) = \div(\cdot)-\iota_{\nabla \phi}(\cdot),$$
and
$$\Delta_{\phi}(\cdot) := \Delta(\cdot)-\nabla_{\nabla \phi}(\cdot).$$
\begin{theorem2}[Cao--Hamilton--Ilmanen \cite{CHI}, Cao--Zhu \cite{CZ}]\label{CZstab}
	Let $(M,g,\phi)$ be a gradient Ricci soliton with potential function $\phi$ and constant $\lambda$. Let $h\in s^{2}(TM)$. Then
	$$
	\frac{d^{2}}{ds^{2}}\nu(g+sh) |_{s=0} = \frac{(2\lambda)^{-1}}{(8\pi\lambda)^{n/2}}\int_{M}\langle N(h),h\rangle e^{-\phi}dV_{g},	
	$$
	where
	\begin{equation}\label{Stabop}
	N(h) = \frac{1}{2}\Delta_{\phi}(h)+\Riem(h,\cdot)+\div^{\ast}\div_{\phi}h+\frac{1}{2}\nabla^{2}v_{h} -\left(\frac{\int_{M}\langle\Ric(g),h\rangle e^{-\phi}dV_{g}}{\int_{M}\mathrm{scal}(g)e^{-\phi}dV_{g}}\right)\Ric(g),
	\end{equation}
	$\mathrm{scal}(g)$ is the scalar curvature of $g$, and $v_{h}$ is the unique solution to
	$$ \Delta_{\phi}v_{h}+\lambda v_{h}=\div_{\phi}\div_{\phi}h. $$
	
\end{theorem2}
As Perelman's $\nu$-entropy is invariant under homothetic rescaling and diffeomorphisms of the metric, we restrict to perturbations given by tensors $h$ satisfying the gauge-fixing conditions
$$ \div_{\phi}(h) =0 \textrm{ and } \int_{M}\langle \Ric(g),h\rangle e^{-\phi}dV_{g}=0.$$
In the Einstein case (where the soliton potential $\phi$ is constant), the stability operator (\ref{Stabop}) restricted to these tensors is given by
\begin{equation*}\label{StabopEin}
N(h) = \frac{1}{2}(\Delta_{L}h+2\lambda h),
\end{equation*}
where $\Delta_{L}$ is the Lichnerowicz Laplacian
\begin{equation}\label{LLap}
\Delta_{L}h = \Delta h +2\Riem(h,\cdot)-\Ric\cdot h-h\cdot \Ric.
\end{equation}
Hence we can state a stability criterion for Einstein metrics in terms of the spectrum of the Lichnerowicz Laplacian.
\begin{definition}[Linear stability of Einstein metrics \cite{CHI}, \cite{CH}]
	Let $(M,g)$ be a compact Einstein manifold satisfying $\Ric(g) =\lambda g$ and let $-\kappa$ be the largest eigenvalue of the Lichnerowicz Laplacian restricted to the space of divergence-free, $g$-orthogonal tensors.
	\begin{enumerate}
		\item If $\kappa > 2\lambda$, $g$ is called \textit{linearly stable}. 
		\item If $\kappa =2\lambda$, $g$ is called \textit{neutrally linearly stable}.
		\item If $\kappa < 2\lambda$, $g$ is called \textit{linearly unstable}.   
	\end{enumerate}
	
\end{definition}
Sesum has related the notions of linear stability and the dynamical stability of the Ricci flow \cite{Sesum}. In particular, an Einstein metric $g_{0}$ is dynamically stable if there exists a $C^{k}$-neighbourhood ($k\geq 3$) $U$ of $g_{0}$ such that the $\lambda$-rescaled Ricci flows
$$\dfrac{\partial g}{\partial t}=-2\Ric(g) +2\lambda g $$
converge to $g_{0}$ for all initial metrics $g \in U$. If no such neighbourhood exists then we say $g_{0}$ is unstable.  This article is concerned with unstable metrics and we note the following result of Sesum.
\begin{proposition}[\cite{Sesum}]
	Let $(M,g)$ be a compact Einstein manifold satisfying ${\Ric(g) =\lambda g}$. If the Lichnerowicz Laplacian has a divergence-free, $g$-orthogonal eigentensor with eigenvalue $-\kappa$ satisfying
	\begin{equation*}
	\kappa < 2\lambda,
	\end{equation*} 
	then $g$ is unstable as a fixed point of the Ricci flow.
\end{proposition}
It is expected that Einstein metrics which are stable under the Ricci flow are quite special. In dimension four Richard Hamilton has conjectured that the only linearly stable examples of positively curved Einstein metrics (or Ricci solitons) are $\mathbb{S}^{4}$ and $\mathbb{CP}^{2}$ with their standard metrics. The Fubini--Study on $\mathbb{CP}^{n}$ metric is neutrally linearly stable as the eigentensors of the Lichnerowicz Laplacian achieve the bound $-2\lambda$. Recent work by Kr\"oncke \cite{KlKr1} showed the surprising result that the Fubini--Study metric on $\mathbb{CP}^{n}$ is not dynamically stable. Various works \cite{DWW}, \cite{HHS}, \cite{HMPAMS}, \cite{HMAGAG1}, \cite{HMPEMS} have provided case-by-case evidence for Hamilton's conjecture but as yet, very little general theory exists. Recent work by Pali has addressed this in the case of K\"ahler--Ricci solitons and K\"ahler--Ricci flow \cite{Pali2}, \cite{Pali1} . By looking at compact symmetric spaces, Cao and He showed that there do exist a wider variety of stable Einstein metrics in higher dimensions \cite{CH}. 
\\
In the case that function $f$ in Equations (\ref{KKeq}) and (\ref{QEMeq}) is constant, we recover the usual notion of a product Einstein metric.  It is well-known that ordinary products can be destabilised by `inflating' one of the factors. The destabilising perturbations we use to prove the main theorems follow a similar idea but of course the presence of the non-constant warping factor $f$ complicates this procedure.\\   
\\
We end this section with a lemma that will prove useful in subsequent calculations.
\begin{lemma}\label{Lem1}
	Let $(M^{n},g)$ be an Einstein manifold with Einstein constant $\lambda>0$ and let $h$ be a divergence-free tensor. Then: 
	\begin{enumerate}
		\item The tensor $h+cg$
		satisfies 
		$$\int_{M}\langle h+cg,g\rangle \ d V_{g} = 0,$$
		where
		$$c = \frac{-\int_{M}\tr(h)dV_{g}}{n\vol(M)}.$$
		\item The stability integral for $h+cg$ is given by
		$$\langle N(h+cg),h+cg\rangle_{L^{2}(g)} = \langle\frac{1}{2}\Delta h +\Riem(h,\cdot),h\rangle_{L^{2}(g)}-\lambda\frac{\left(\int_{M}\tr(h) dV_{g} \right)^{2}}{n\vol(M)}.$$
	\end{enumerate}
	
\end{lemma}
\begin{proof}
	(1) is a trivial calculation. To see  (2) note that by Equation  \ref{Stabop}, the stability operator $N$ is given by
	$$N(h+cg)=\frac{1}{2}\Delta(h+cg)+\Riem(h+cg,\cdot).$$
	Hence we see
	$$\langle N(h+cg),h+cg\rangle_{L^{2}(g)} = \langle\frac{1}{2}\Delta (h+cg) +\Riem(h+cg,\cdot),h+cg\rangle_{L^{2}(g)}.$$
	Using the fact that $\Delta(cg)=0$ and, as $g$ is an Einstein metric, 
	$$ \langle\Riem(h+cg,\cdot),cg\rangle_{L^{2}(g)} = c\lambda \int_{M}\tr(h+cg)dV_{g}=0,$$
	we obtain
	$$\langle N(h+cg),h+cg\rangle_{L^{2}(g)} = \langle\frac{1}{2}\Delta h+\Riem(h,\cdot),h\rangle_{L^{2}(g)}+ \langle \Riem(cg,\cdot),h\rangle_{L^{2}(g)}.$$
	The claim follows by noting $\Riem(cg,\cdot) = c\Ric(g)=c\lambda g$ and the value of $c$ from part (1).
\end{proof}

\subsection{Black hole stability}
In \cite{GH} the authors developed the stability theory of generalised Schwarzschild--Tangherlini black holes. These are metrics of the form
\begin{equation*} 
d\hat{s}^{2} = -\left[1-\left(\frac{l}{r}\right)^{d-1}\right]dt^{2}+\frac{dr^{2}}{\left[ 1-\left(\frac{l}{r}\right)^{d-1}\right]} + r^{2}ds^{2}_{d},
\end{equation*}
where $l$ is a constant and $ds^{2}_{d}$ is the metric on a $d$ dimensional compact Einstein manifold normalised so that its Einstein constant is $d-1$. They found a stability criterion involving the spectrum of the Lichnerowicz Laplacian restricted to the divergence-free (called transverse in the physics literature), trace-free tensors on the Einstein manifold. We state their stability criterion with respect to our convention that the ordinary Laplacian has a non-positive spectrum (this is the opposite convention to that taken in \cite{GH}).
\begin{proposition}[Black hole linear stability \cite{GH}]\label{BHstab}
	Let $(M^{n},g)$ be a compact Einstein manifold satisfying $\Ric(g) =\lambda g$. Then the associated Schwarzschild--Tangherlini black hole is linearly unstable if the Lichnerowicz Laplacian has a divergence-free, trace-free eigentensor with eigenvalue $-\kappa$  satisfying
	\begin{equation*}
	\kappa < \frac{\lambda}{n-1}\left(4-\frac{(5-n)^{2}}{4}\right)=\frac{(9-n)\lambda}{4}.
	\end{equation*}
\end{proposition} 
Note that if an Einstein metric is unstable in the black hole sense, then  it is unstable as a fixed point of the Ricci flow. We also note that we require genuinely trace-free perturbations in this definition rather than perturbations which are $L^{2}$-orthogonal to the metric (i.e. the integral of the trace of the perturbation is zero).\\
\\
In \cite{GHP}, Gibbons, Hartnoll and Pope investigated the linear stability of the B\"ohm warped product metrics.  They proved that the  B\"ohm metrics (or rather the associated  Schwarzschild--Tangherlini black holes) on $\mathbb{S}^{3}\times \mathbb{S}^{m}$ for $m=2,3$ are unstable. Their proof used the cohomogeneity one symmetry that the B\"ohm metrics exhibit. Our Theorem \ref{thmE} is a generalisation of this result to an arbitrary warped product metric on these spaces. 
\subsection{A  heuristic for destabilising one parameter families of warped products}
The methods for proving the Theorems \ref{thmA} and \ref{thmB} are inspired by considering how to destabilise a product gradient Ricci soliton. Let    $(B,\bar{g},\phi)$ be a gradient Ricci soliton satisfying
$$\Ric(\bar{g})+\overline{\nabla}^{2}\phi=\lambda \bar{g},$$
and  let  $(F,\tilde{g})$ be an Einstein manifold satisfying
$$\Ric(\tilde{g}) = \lambda\tilde{g},$$
where $\lambda>0$. Then the metric ${g=\bar{g}\oplus \tilde{g}}$ is a gradient Ricci soliton on $M=B\times F$ with potential function $\phi \circ\pi_{B}$ (as usual, we will drop the reference to the projection from now on and also denote this function $\phi$).\\
\\
As mentioned previously, it is natural to use gauge-fixed tensors to destabilise. For a Ricci soliton, this means choosing tensors which are $\div_{\phi}$-free and $L^{2}$-orthogonal, with respect to the weighted volume form $e^{-\phi}dV_{\bar{g}}$, to the Ricci tensor. On a product soliton there are two natural tensors satisfying the gauge-fixing conditions. The first is the tensor
$$h_{1}=e^{\phi}\left(\frac{\bar{g}}{n}\oplus -\frac{1}{m}\tilde{g}\right).$$ 
The second is the tensor
$$ h_{2} = \Ric(\bar{g})\oplus c\tilde{g},$$
with the constant $c$ chosen so that
$$\int_{M}\langle \Ric(g),h_{2}\rangle e^{-\phi}dV_{\bar{g}}=0.$$ 
One can compute the stability integral in Theorem \ref{CZstab} for each of the perturbations. 
\begin{proposition}\label{RDP1}
	Let $(B\times F,g=\bar{g}\oplus \tilde{g},\phi)$ be a gradient product Ricci soliton and let $h_{1}$ and $h_{2}$ be as above. Then
	\begin{align*}
	\langle N(h_{1}),h_{1}\rangle_{L^{2}(e^{-\phi}dV_{g})} =  \vol (F)\int_{B}\left(-\left(\frac{1}{2n}+\frac{1}{2m} -\frac{1}{n^{2}}\right)|\overline{\nabla} e^{\phi}|^{2} +\left(\frac{1}{n}+\frac{1}{m}\right)\lambda e^{2\phi}\right)e^{-\phi}dV_{\bar{g}}, 
	\end{align*}
	and
	$$\langle N(h_{2}),h_{2}\rangle_{L^{2}(e^{-\phi}dV_{g})}  = \lambda \|h_{2}\|_{L^{2}(e^{-\phi}dV_{g})}^{2}.$$
\end{proposition}

\begin{proof}
As $h_1$ satisfies the gauge-fixing conditions ${\div_{\phi}(h_{1})=0}$ and ${\langle \Ric(g),h_{1}\rangle_{L^{2}(e^{-\phi}dV_{g})}=0}$, equation (\ref{Stabop}) reduces to
$$N(h_{1}) = \frac{1}{2}\Delta_{\phi}h_{1}+\Riem(h_{1},\cdot).$$
As we are working with the rescaled volume form $e^{-\phi}dV_{g}$ we can integrate the first term by parts and get
$$\frac{1}{2}\langle\Delta_{\phi}h_{1},h_{1}\rangle_{L^{2}(e^{-\phi}dV_{g})} = \int_{M}-\frac{1}{2}|\nabla h_{1}|^{2}e^{-\phi}dV_{g}.$$
We then compute
$$|\nabla h_{1}|^{2} = |\overline{\nabla}e^{\phi}|^{2}|\frac{\bar{g}}{n}\oplus\frac{-\tilde{g}}{m}|^{2} =  |\overline{\nabla}e^{\phi}|^{2}\left(\frac{1}{n}+\frac{1}{m} \right).$$	
Hence
$$\frac{1}{2}\langle\Delta_{\phi}h_{1},h_{1}\rangle_{L^{2}(e^{-\phi}dV_{g})} = \vol(F)\int_{B}-|\overline{\nabla}e^{\phi}|^{2}\left(\frac{1}{2n}+\frac{1}{2m} \right)e^{-\phi}dV_{\bar{g}}.$$
The curvature operator term is given by
$$\Riem(h_{1},\cdot) = e^{\phi}\left(\frac{\Ric(\bar{g})}{n}\oplus -\lambda\frac{\tilde{g}}{m}\right),$$
thus, using the fact $\bar{g}$ is a gradient Ricci soliton,
$$\langle\Riem(h_{1},\cdot),h_{1}\rangle = \left(-\frac{\overline{\Delta}\phi}{n^{2}} +\lambda\left(\frac{1}{n}+\frac{1}{m}\right)\right)e^{2\phi}.$$
We note the identity $|\overline{\nabla}\phi|^{2}e^{\phi} =|\overline{\nabla}e^{\phi}|^{2}e^{-\phi} $ and so integrating by parts yields
$$\langle\Riem(h_{1},\cdot),h_{1}\rangle_{L^{2}(e^{-\phi}dV_{g})} =\vol(F) \int_{B}\left(\frac{|\overline{\nabla}e^{\phi}|^{2}}{n^{2}}+\lambda e^{2\phi}\left(\frac{1}{n}+\frac{1}{m}\right)\right)e^{-\phi}dV_{\bar{g}}.$$ 
Combining the Laplacian and curvature operator terms yields the identity for the $h_{1}$ tensor.\\
\\
For the variation $h_{2}$ we note that the results of Cao and Zhu \cite{CZ} show that $h_{2}$ is ${\div_{\phi}\mathrm{-free}}$ and of course, from the choice of $c$, is $L^{2}(e^{-\phi}dV_{g})$-orthogonal to $\Ric(g)$. Hence we compute as with the tensor $h_{1}$,
$$\frac{1}{2}\Delta_{\phi}h_{2}+\Riem(h_{2},\cdot) = \left(\frac{1}{2}\overline{\Delta}_{\phi}(\Ric(\bar{g}))+\overline{\Riem}(\Ric(\bar{g}),\cdot)\right)\oplus c\left(\frac{1}{2}\widetilde{\Delta}(\tilde{g})+\widetilde{\Riem}(\tilde{g},\cdot)) \right).$$
We again use the results of Cao and Zhu \cite{CZ} who show that (as $\bar{g}$ is a gradient Ricci soliton) 
$$\frac{1}{2}\overline{\Delta}_{\phi}(\Ric(\bar{g}))+\overline{\Riem}(\Ric(\bar{g}),\cdot) = \lambda \Ric(\bar{g}).$$ 
Finally, as $\tilde{g}$ is Einstein,  
$$\frac{1}{2}\widetilde{\Delta}(\tilde{g})+\widetilde{\Riem}(\tilde{g},\cdot)) = \lambda c \tilde{g}, $$
and we see $N(h_{2}) = \lambda h_{2}$. The result follows.
\end{proof}

While it is clear from Proposition \ref{RDP1} that the perturbation $h_{2}$ always destabilises, it is not clear that this is true for $h_{1}$. Note that
$$\int_{B}|\overline{\nabla}e^{\phi}|^{2}e^{-\phi}dV_{\bar{g}} = \frac{1}{2}\int_{B}\langle \overline{\nabla} \phi, \overline{\nabla} e^{2\phi}\rangle e^{-\phi}dV_{\bar{g}}.$$
It is well-known \cite{CaoRPRS} that it is also possible to normalise $\phi$ so that $$\overline{\Delta}_{\phi}\phi=-2\lambda \phi.$$
In this case
$$\int_{B}|\overline{\nabla}e^{\phi}|^{2}e^{-\phi}dV_{\bar{g}}  = \lambda\int_{B}\phi e^{\phi}dV_{\bar{g}},$$
and the perturbation $h_{1}$ is destabilising if
\begin{equation*}\label{GHPRSineq}
\int_{B} \phi e^{\phi}dV_{\bar{g}}<\frac{2n(m+n)}{n^2 + mn - 2m}\int_{B}e^{\phi}dV_{\bar{g}}.
\end{equation*}
Such an inequality does not in general hold for functions $\psi$ satisfying
$$\int_{B}\psi e^{-\psi}dV_{\bar{g}}=0,$$
and so  it is not possible to conclude that $h_{1}$ is a destabilising perturbation of a product Ricci soliton.\\
\\
Suppose that there is a sequence of quasi-Einstein metrics $(M,\bar{g}_{i},f_{i},m_{i})$ solving Equations (\ref{QEMeq}) and (\ref{KKeq}) for some fixed $\lambda$ and $\mu_{i}$ (we are always free to fix $\lambda$ by rescaling the $g_{i}$). Setting $\phi_{i} = -m_{i}\log f_{i}$, Equation (\ref{QEMeq}) becomes
\begin{equation*}
\Ric(\bar{g}_{i})+\overline{\nabla}^{2}\phi_{i}-\frac{d\phi_{i}\otimes d\phi_{i}}{m_{i}}=\lambda \bar{g}_{i}.
\end{equation*}
If as $i\rightarrow \infty$, $m_{i}\rightarrow \infty$, the metrics $\bar{g}_{i}$ and the functions $\phi_{i}$ converge (as we  are only outlining some heuristic reasoning we do not make this notion of convergence precise), then the limiting metric $\bar{g}_{\infty}$ and function $\varphi$ solve the gradient Ricci soliton equation
\begin{equation*}
\Ric(\bar{g}_{\infty})+\overline{\nabla}^{2}\varphi=\lambda \bar{g}_{\infty}.
\end{equation*} 
For large values of $m_{i}$, we can make the following approximations
$$
\bar{g}_{i}\approx \bar{g}_{\infty}, \quad f_{i} \approx 1,  \textrm{ and }  f_{i}^{m_{i}}\approx e^{-\varphi}.
$$ 
 Hence the warped product Einstein metric  $\bar{g}_{i}\oplus f_{i}^{2}\tilde{g}$ can be approximated by the metric $\bar{g}_{\infty}\oplus \tilde{g}$. We shall see in Section \ref{sec:5} that $\mu_{i} \approx \lambda$, and so, for large values of $m_{i}$, the metric $\bar{g}_{i}\oplus \tilde{g}$ is almost a product Ricci soliton.\\
\\
In Sections \ref{sec:4} and \ref{sec:5}  we define two different tensors. The first are the GHP variations (Definition \ref{GHPvardef}) which are the analogues of the tensor $h_{1}$. The fact that $h_{1}$ is not obviously universally destabilising goes some way to explain why using GHP variations fails to destabilise warped products in all but the lowest dimensions. The Ricci variation (Definition \ref{Ricvardef}) is the analogue of $h_{2}$ and Theorem \ref{thmB} could be paraphrased as saying that, providing the metrics $\bar{g}_{i}$ and $\bar{g}$ are close to each other, the fact that $h_{2}$ is universally destabilising means the Ricci variation also destabilises the warped product for large values of $i$. 
\section{Geometric operators for warped products}
\label{sec:3}
In this section we collect some useful identities that are used in the proof of the main theorems. All of the theorems involve choosing a potentially destabilising tensor ${h \in s^{2}(TM)}$ and then computing the Rayleigh quotient
\begin{equation*}
\frac{\int_{M}\langle \Delta_{L}h,h\rangle dV_{g}}{\int_{M}|h|^{2}dV_{g}},
\end{equation*}
which provides a lower bound for the least negative eigentensor of $\Delta_{L}$. The class of destabilising tensors that we consider can be written in the form
\begin{equation}\label{wptens}
h = \overline{h}\oplus \psi \tilde{h},
\end{equation}
where ${\overline{h} \in s^{2}(TB)}$, ${\tilde{h} \in s^{2}(TF)}$ and ${\psi \in C^{\infty}(B)}$.\\
\\
One fundamental calculation is of the Christoffel symbols for the Levi-Civita connection of a warped product metric.
\begin{lemma}[Christoffel symbols of $g$]\label{WPcon}
	Let $M = B\times F$ be a product manifold and let $g = \overline{g}\oplus f^{2}\tilde{g}$ be a warped product metric on $M$. Then the Christoffel symbols for the Levi-Civita connection of $g$ are given by:
	\begin{align*}
	\Gamma_{ab}^{c} &= \overline{\Gamma}_{ab}^{c},\\
	\Gamma_{\alpha \beta}^{c} &= -\tilde{g}_{\alpha \beta}f\overline{g}^{cd}(\overline{\nabla}_{d}f),\\
	\Gamma_{a\beta}^{\gamma} & = (\overline{\nabla}_{a}\log f)\delta^{\gamma}_{\beta},\\
	\Gamma_{\alpha\beta}^{\gamma} & = \tilde{\Gamma}_{\alpha\beta}^{\gamma}.
	\end{align*}
	All other symbols are zero.
\end{lemma}
It will also be useful to know an explicit form of the curvature tensor.
\begin{lemma}[Curvature tensor of $g$]\label{CurvTenWP}
	Let $M = B\times F$ be a product manifold and let $g = \overline{g}\oplus f^{2}\tilde{g}$ be a warped product metric on $M$. Then the curvature tensor for $g$ can be described by:
	\begin{align*}
	R_{abcd} & = \overline{R}_{abcd},\\
	R_{a\beta\gamma d} & =f\tilde{g}_{\beta\gamma}(\overline{\nabla}^{2}f)_{ad},\\
	R_{\alpha\beta\gamma\delta} & =f^{2}\widetilde{\Riem}_{\alpha\beta\gamma\delta}-f^{2}|\overline{\nabla}f|^{2}(\tilde{g}_{\alpha\gamma}\tilde{g}_{\beta\delta}-\tilde{g}_{\alpha\delta}\tilde{g}_{\beta\gamma}) .
	\end{align*}
	All other components are zero.
\end{lemma}
	 

As mentioned in Section \ref{sec:2}, one need only check stability on divergence-free tensors. The next lemma computes the divergence of tensors of the form (\ref{wptens}).
\begin{lemma}[Divergence of $h$]\label{WPdiv}
	Let $(B\times F^{m},\overline{g}\oplus f^{2}\tilde{g})$ be a warped product manifold and let $h$ be of the form \textrm{(\ref{wptens})}.  Then 
	\begin{align*}
	\div(h)(\cdot) = & \overline{\div}(\overline{h})(\cdot)+m\overline{h}(\overline{\nabla} \log f, \cdot)\\
	&-f^{-2}\psi (\tilde{\tr}(\tilde{h}))d \log f(\cdot)+f^{-2}\psi \widetilde{\div}(\tilde{h})(\cdot).
	\end{align*}
\end{lemma}
\begin{proof}
	In coordinates, as $g^{a\beta}=0$, we have 
	$$\div(h)_{C} = g^{AB}(\nabla_{A}h)(\partial_{B},\partial_{C}) = \bar{g}^{ab}(\nabla_{a}h)(\partial_{b},\partial_{C}))+f^{-2}\tilde{g}^{\alpha\beta}(\nabla_{\alpha}h)(\partial_{\beta},\partial_{C}).$$
	Using  Lemma \ref{WPcon} and the fact $h_{a\beta}=0$ we see
	$$(\nabla_{a}h)(\partial_{b},\partial_{\gamma})=0.$$
	Hence ${\bar{g}^{ab}(\nabla_{a}h)(\partial_{b},\partial_{C})) = \overline{\div}(\bar{h})(\partial_{C})}$. For the next term we again use the Christoffel symbols of Lemma \ref{WPcon}
	$$ f^{-2}\tilde{g}^{\alpha\beta}(\nabla_{\alpha}h)(\partial_{\beta},\partial_{c})= f^{-2}\tilde{g}^{\alpha\beta}(\nabla_{\alpha}h_{\beta c}-\Gamma_{\alpha \beta}^{d}h_{cd}-\Gamma_{\alpha c}^{\delta}h_{\beta \delta}),
	$$
	which as, $h_{\beta c} =0$, ${\Gamma_{\alpha \beta}^{d} = -\tilde{g}_{\alpha \beta}f\overline{g}^{de}(\overline{\nabla}_{e}f)}$, and $\Gamma_{\alpha c}^{\delta} = (\overline{\nabla}_{c}\log f)\delta^{\delta}_{\alpha},$ yields
	$$f^{-2}\tilde{g}^{\alpha\beta}(\nabla_{\alpha}h)(\partial_{\beta},\partial_{c})=m\bar{h}(\overline{\nabla}\log f,\partial_{c})-f^{-2}\psi\tilde{\tr}(\tilde{h})(\overline{\nabla}_{c}\log f).$$
	Finally, we consider
	$$f^{-2}\tilde{g}^{\alpha\beta}(\nabla_{\alpha}h)(\partial_{\beta},\partial_{\gamma}) = f^{-2}\tilde{g}^{\alpha\beta}(\nabla_{\alpha}h_{\beta \gamma}-\Gamma_{\alpha \beta}^{\delta}h_{\gamma \delta}-\Gamma^{\delta}_{\alpha \gamma}h_{\beta \delta} ).
	$$
	Hence as $h_{\alpha \beta} = \psi\tilde{h}_{\alpha \beta}$,
	$$
	f^{-2}\tilde{g}^{\alpha\beta}(\nabla_{\alpha}h)(\partial_{\beta},\partial_{\gamma}) = f^{-2}\psi\widetilde{\div}(\tilde{h})(\partial_{\gamma}).
	$$
	Combining each piece finishes the proof.
\end{proof}

To break down the calculation of the Lichnerowicz Laplacian, we first compute the connection Laplacian of the tensors $h$.
\begin{lemma}[Connection Laplacian]\label{LapLem}
	Let $(B\times F^{m},\overline{g}\oplus f^{2}\tilde{g})$ be a warped product manifold and let $h$ be of the form \textrm{(\ref{wptens})}. Then 
	\begin{align*}
	\Delta h & =   \overline{\Delta}\bar{h}+2f^{-2}\psi\tilde{\tr}(\tilde{h})(d\log f \otimes d\log f)\\
	& - m(d\log f\otimes \iota_{\overline{\nabla} \log f}\overline{h}+\iota_{\overline{\nabla} \log f}\overline{h}\otimes d\log f-(\overline{\nabla}_{\overline{\nabla}\log f}\overline{h}))\\
	& +(\overline{\Delta}\psi - 2\psi\overline{\Delta}\log f+(m-4)\overline{g}(\overline{\nabla}\psi,\overline{\nabla}\log f)+2(1-m)\psi|\overline{\nabla}\log f|^{2})\tilde{h}\\
	& +  f^{-2}\psi\widetilde{\Delta}\tilde{h}+2\overline{h}(\overline{\nabla}f,\overline{\nabla}f)\tilde{g}\\
	&-2\psi f^{-2}(d\log f\otimes \widetilde{\div}(\tilde{h})+\widetilde{\div}(\tilde{h})\otimes d\log f).
	\end{align*}
	
\end{lemma}

\begin{proof}
We begin by noting that, for any tensor $T$, ${\nabla^{2}_{A,B}T = \nabla_{A}\nabla_{B}T-\nabla_{\nabla_{A}B}T}$. We need to compute
$$\Delta h = g^{AB}\nabla^{2}_{A,B}h =\bar{g}^{ab}\nabla^{2}_{a,b}h+f^{-2}\tilde{g}^{\alpha\beta}\nabla^{2}_{\alpha,\beta}h.$$
The proof proceeds by using the coordinate description of the connection given in Lemma \ref{WPcon}.  We will compute each separate part of $\Delta h$.\\
\\
We begin by computing $(\Delta h)_{cd}$. We can easily verify that
\begin{equation}\label{proofeq11}
(\nabla_{a}\nabla_{b}h)_{cd} = (\overline{\nabla}_{a}\overline{\nabla}_{b}\bar{h})_{cd}
\end{equation}
and  
\begin{equation}\label{proofeq12}
(\nabla_{\nabla_{a}b}h)_{cd} = (\overline{\nabla}_{\overline{\nabla}_{a}b}\bar{h})_{cd}. 
\end{equation} 
To compute the $(\nabla^{2}_{\alpha,\beta}h)_{cd}$ terms, we expand
$$(\nabla_{\alpha}\nabla_{\beta}h)_{cd} = \nabla_{\alpha}(\nabla_{\beta}h)_{cd}-(\nabla_{\beta}h)(\nabla_{\alpha}{c},\partial_{d}) -(\nabla_{\beta}h)(\partial_{c},\nabla_{\alpha}{d}), $$
which yields
$$(\nabla_{\alpha}\nabla_{\beta}h)_{cd}  = \nabla_{\alpha}(\nabla_{\beta}h)_{cd}-(\overline{\nabla}_{c}\log f)(\nabla_{\beta}h)_{\alpha d} -(\overline{\nabla}_{d}\log f)(\nabla_{\beta}h)_{c \alpha}.$$
A straightforward check shows $$(\nabla_{\beta}h)_{cd}=0, $$ and, for example,
$$(\nabla_{\beta}h)_{\alpha d}= \tilde{g}_{\alpha\beta}f\bar{h}(\overline{\nabla}f,\partial_{d})-\psi(\overline{\nabla}_{d}\log f)\tilde{h}_{\alpha\beta}.$$
Putting all this together we obtain
\begin{equation}\label{proofeq13}
(\nabla_{\alpha}\nabla_{\beta}h)_{cd} = 2\psi\tilde{h}_{\alpha \beta}(\overline{\nabla}_{c}\log f\overline{\nabla}_{d}\log f) - \tilde{g}_{\alpha\beta}f\left((\overline{\nabla}_{c}\log f)\bar{h}(\overline{\nabla}f,\partial_{d})+(\overline{\nabla}_{d}\log f)\bar{h}(\overline{\nabla}f,\partial_{c})\right).
\end{equation}
In a similar manner we see that
\begin{equation}\label{proofeq14}
(\nabla_{\nabla_{\alpha}\beta}h)_{cd} = -\tilde{g}_{\alpha\beta}f(\overline{\nabla}_{\overline{\nabla}f}\bar{h})_{cd}.\end{equation}
Combining Equations (\ref{proofeq11}) and (\ref{proofeq12}), and (\ref{proofeq13}) and (\ref{proofeq14}), then taking traces yields
  	\begin{align*}
  (\Delta h)_{cd} & =   (\overline{\Delta}\bar{h})_{cd}+2f^{-2}\psi\tilde{\tr}(\tilde{h})(d\log f \otimes d\log f)_{cd}\\
  & - m(d\log f\otimes \iota_{\overline{\nabla} \log f}\overline{h}+\iota_{\overline{\nabla} \log f}\overline{h}\otimes d\log f-(\overline{\nabla}_{\overline{\nabla}\log f}\overline{h}))_{cd}.\\
  \end{align*}
  \\
  We now compute $(\Delta h)_{\gamma\delta}$.  The term
  $$(\nabla_{a}\nabla_{b} h)_{\gamma \delta} = \nabla_{a}(\nabla_{b} h)_{\gamma \delta} - (\nabla_{b} h)(\nabla_{a}\gamma,\partial_{\delta})- (\nabla_{b} h)(\partial_{\gamma},\nabla_{a}\delta),$$
  is given by
  $$(\nabla_{a}\nabla_{b} h)_{\gamma \delta} = \nabla_{a}(\nabla_{b} h)_{\gamma \delta} -2(\overline{\nabla}_{a}\log f)(\nabla_{b} h)_{\gamma \delta}.$$
  We expand
  $$(\nabla_{b} h)_{\gamma \delta} = (\overline{\nabla}_{b}\psi-2\psi\overline{\nabla}_{b}\log f)\tilde{h}_{\gamma\delta},$$
  which yields
  $$(\nabla_{a}\nabla_{b} h)_{\gamma \delta}  =$$ 
  \begin{equation}\label{proofeq21}
  (\overline{\nabla}_{a}\overline{\nabla}_{b}\psi-2(\overline{\nabla}_{a}\psi\overline{\nabla}_{b}\log f+\overline{\nabla}_{b}\psi\overline{\nabla}_{a}\log f)-2\psi\overline{\nabla}_{a}\overline{\nabla}_{b}\log f+4\psi\overline{\nabla}_{a}\log f\overline{\nabla}_{b}\log f)\tilde{h}_{\gamma\delta}.
   \end{equation}
  Similarly, we obtain
  \begin{equation}\label{proofeq22}
  (\nabla_{\nabla_{a}b} h)_{\gamma \delta}  = \left(\overline{\nabla}_{\overline{\nabla}_{a}b}\psi-2\psi\overline{\nabla}_{\overline{\nabla}_{a}b}\log f \right)\tilde{h}_{\gamma\delta}.
  \end{equation}
  
 To compute $(\nabla_{\alpha}\nabla_{\beta}h)_{\delta\gamma}$ we expand, 
 $$ (\nabla_{\alpha}\nabla_{\beta}h)_{\delta\gamma}= \nabla_{\alpha}(\nabla_{\beta}h)_{\gamma\delta}-(\nabla_{\beta}h)(\nabla_{\alpha}\gamma,\partial_{\delta})-(\nabla_{\beta}h)(\partial_{\gamma}, \nabla_{\alpha}\delta),$$
 $$ = \nabla_{\alpha}(\nabla_{\beta}h)_{\gamma\delta} - (\nabla_{\beta}h)(\widetilde{\nabla}_{\alpha}\gamma,\partial_{\delta})-(\nabla_{\beta}h)(\partial_{\gamma}, \widetilde{\nabla}_{\alpha}\delta)+f\tilde{g}_{\alpha \gamma}(\nabla_{\beta}h)(\overline{\nabla}f,\partial_{\delta})+f\tilde{g}_{\alpha \delta}(\nabla_{\beta}h)(\partial_{\gamma},\overline{\nabla}f).$$
 Noting that ${(\nabla_{\beta} h)_{\delta \gamma}  =\psi(\widetilde{\nabla}_{\beta} \tilde{h})_{\delta \gamma}}$ and, for example, 
$$ (\nabla_{\beta}h)(\overline{\nabla}f,\partial_{\delta}) = -f|\overline{\nabla} \log f|^{2}\psi \tilde{h}_{\beta\delta}+f\tilde{g}_{\beta\delta}\bar{h}(\overline{\nabla}f,\overline{\nabla}f), $$
we obtain
$$ (\nabla_{\alpha}\nabla_{\beta}h)_{\gamma\delta} = $$
\begin{equation}\label{proofeq23}
\psi\widetilde{\nabla}_{\alpha}(\widetilde{\nabla}_{\beta}\tilde{h})_{\gamma\delta}+f^{2}(\tilde{g}_{\alpha\gamma}\tilde{g}_{\beta\delta}+\tilde{g}_{\alpha\delta}\tilde{g}_{\beta\gamma})\bar{h}(\overline{\nabla}f,\overline{\nabla}f)-f^{2}|\overline{\nabla} \log f|^{2}\psi((\tilde{g}_{\alpha\gamma}\tilde{h}_{\beta\delta}+\tilde{g}_{\alpha\delta}\tilde{h}_{\beta\gamma}).
\end{equation}
In a similar manner we see that
\begin{equation}\label{proofeq24}
(\nabla_{\nabla_{\alpha}\beta}h)_{\gamma\delta} = \psi(\widetilde{\nabla}_{\widetilde{\nabla}_{\alpha}\beta}\tilde{h})_{\gamma\delta} + \tilde{g}_{\alpha\beta}f^{2}\left(2\psi|\overline{\nabla}\log f|^{2}-\bar{g}(\overline{\nabla} \log f,\overline{\nabla} \psi)\right)\tilde{h}_{\gamma\delta}.
\end{equation}
Combining Equations (\ref{proofeq21}) and (\ref{proofeq22}), and (\ref{proofeq23}) and (\ref{proofeq24}), then taking traces yields
\begin{align*}
(\Delta h)_{\gamma\delta}  & =\\
& (\overline{\Delta}\psi - 2\psi\overline{\Delta}\log f+(m-4)\overline{g}(\overline{\nabla}\psi,\overline{\nabla}\log f)+2(1-m)\psi|\overline{\nabla}\log f|^{2})\tilde{h}_{\gamma\delta}\\
& +  f^{-2}\psi(\widetilde{\Delta}\tilde{h})_{\gamma\delta}+2\overline{h}(\overline{\nabla}f,\overline{\nabla}f)\tilde{g}_{\gamma\delta}. 
\end{align*}
\\
Finally, we calculate $(\Delta h)_{c\delta}$.  As $(\nabla_{b} h )_{c\delta}=0$, we find
$$(\nabla_{a}\nabla_{b} h)_{c\delta} = 0 \qquad \textrm{ and } \qquad (\nabla_{\nabla_{a}b} h)_{c\delta}  = 0.$$
Expanding $(\nabla_{\alpha}\nabla_{\beta} h)_{c\delta}$ yields
$$(\nabla_{\alpha}\nabla_{\beta} h)_{c\delta} = \nabla_{\alpha}(\nabla_{\beta}h)_{c\delta}-(\overline{\nabla}_{c}\log f)(\nabla_{\beta}h)_{\alpha \delta}-(\nabla_{\beta}h)(\partial_{c},\widetilde{\nabla}_{\alpha}\delta)+f\tilde{g}_{\alpha\delta}(\nabla_{\beta}h)(\partial_{c},\overline{\nabla}f).$$
Computing each term yields
$$(\nabla_{\beta}h)_{c\delta} = -\overline{\nabla}_{c}\log f \psi\tilde{h}_{\beta\delta}+f\tilde{g}_{\beta\delta}\bar{h}(\partial_{c},\overline{\nabla}f),$$
$$(\nabla_{\beta}h)_{\alpha\delta} = \psi(\widetilde{\nabla}_{\beta}\tilde{h})_{\alpha\delta},$$
$$ (\nabla_{\beta}h)(\partial_{c},\widetilde{\nabla}_{\alpha}\delta) =-(\overline{\nabla}_{c} \log f)\psi\tilde{h}(\partial_{\beta},\widetilde{\nabla}_{\alpha}\delta)+f\tilde{g}(\partial_{\beta},\widetilde{\nabla}_{\alpha}\delta )\bar{h}(\partial_{c},\overline{\nabla} f),$$
and
$$(\nabla_{\beta}h)(\partial_{c},\overline{\nabla} f) =0.$$
We also have
$$(\nabla_{\nabla_{{\alpha}}\beta}h)_{c\delta} =-(\overline{\nabla}_{c}\log f)\psi \tilde{h}(\widetilde{\nabla}_{\alpha}\beta,\partial_{\delta})+f\tilde{g}(\widetilde{\nabla}_{\alpha}\beta,\partial_{\delta})\bar{h}(\partial_{c},\overline{\nabla}f). $$
Putting this all together, we have
$$(\nabla^{2}_{\alpha,\beta} h)_{c\delta} = -(\overline{\nabla}_{c}\log f)\psi\left((\widetilde{\nabla}_{\alpha}\tilde{h})_{\beta\delta}+(\widetilde{\nabla}_{\beta}\tilde{h})_{\alpha\delta}\right)+f\bar{h}(\overline{\nabla}f,\partial_{c})(\widetilde{\nabla}_{\alpha}\tilde{g})_{\beta\delta}.$$
Taking traces and noting that $\widetilde{\nabla}\tilde{g}=0$ gives 
$$(\Delta h)_{c\delta} = -2\psi f^{-2} (\overline{\nabla}_{c}\log f)\widetilde{\div}(\tilde{h})_{\delta}.$$
\end{proof}

The other component of the Lichnerowicz Laplacian is the curvature operator.
\begin{lemma}[Curvature operator]\label{COLem}
	Let $(B\times F^{m},\overline{g}\oplus f^{2}\tilde{g})$ be a warped product manifold. and let $h$ be of the form (\ref{wptens}). Then 
	
	\begin{align*}
	\Riem(h,\cdot) & =  \overline{\Riem}(\overline{h},\cdot)-f^{-3}\psi\tilde{\tr}(\tilde{h})(\overline{\nabla}^{2}f)+f^{-2}\psi\widetilde{\Riem}(\tilde{h},\cdot)\\
	&-\psi|\overline{\nabla} \log f|^{2}(\tilde{\tr}(\tilde{h})\tilde{g}-\tilde{h})-f\langle \overline{\nabla}^{2}f,\bar{h}\rangle_{\bar{g}}\tilde{g}.
	\end{align*}
\end{lemma}
\begin{proof}
	The curvature operator is given by
	$$\Riem(h,\cdot)_{AB} = R_{ACBD}h^{CD},$$
	which, given the form of $h$, can be written
	$$\Riem(h,\cdot)_{AB} =R_{AcBd}\bar{h}^{cd}+\psi f^{-4} R_{A\gamma B\delta}\tilde{h}^{\gamma\delta}.$$
	(The factor of $f^{-4}$ appears as we are raising two indices of $\tilde{h}$). We now use the explicit components of the curvature tensor given in Lemma \ref{CurvTenWP}.\\
	\\
	As $R_{\alpha\beta\gamma d} = R_{abc\delta}=0$ we see that $\{A,B\} = \{a,b\}$ or $\{A,B\} = \{\alpha,\beta\}.$
	For the former set we obtain
	$$\Riem(h,\cdot)_{ab} = R_{acbd}\bar{h}^{cd}+\psi f^{-4} R_{a\gamma b\delta}\tilde{h}^{\gamma\delta} = \bar{R}_{acbd}\bar{h}^{cd}-\psi f^{-4}fg_{\gamma \delta}(\overline{\nabla}^{2}f)_{ab}\tilde{h}^{\gamma\delta},$$
	$$ = \overline{\Riem}(\bar{h},\cdot)_{ab}-f^{-3}\psi\tilde{tr}(\tilde{h})(\overline{\nabla}^{2}f)_{ab}.
	$$
	For the latter set we obtain
	$$\Riem(h,\cdot)_{\alpha\beta} = R_{\alpha c\beta d}\bar{h}^{cd}+\psi f^{-4} R_{\alpha\gamma \beta\delta}\tilde{h}^{\gamma\delta},$$ 
	$$= -f\tilde{g}_{\alpha \beta}(\overline{\nabla}^{2}f)_{cd}\bar{h}^{cd}+\psi f^{-2}\widetilde{R}_{\alpha\gamma\beta\delta}\tilde{h}^{\gamma\delta}-\psi f^{-2}|\overline{\nabla} f|^{2}(\tilde{g}_{\alpha\beta}\tilde{g}_{\gamma\delta}-\tilde{g}_{\alpha\gamma}\tilde{g}_{\beta\delta})\tilde{h}^{\gamma\delta}.$$
	This gives
	$$\Riem(h,\cdot)_{\alpha\beta} = -f\langle \overline{\nabla}^{2}f,\bar{h}\rangle_{\bar{g}}\tilde{g}_{\alpha\beta}+\psi f^{-2}\widetilde{\Riem}(\tilde{h},\cdot)_{\alpha\beta}-\psi|\overline{\nabla}\log f|^{2}(\tilde{\tr}(\tilde{h})\tilde{g}_{\alpha\beta}-\tilde{h}_{\alpha\beta}),$$
	and the result now follows.
\end{proof}

\section{The proofs of Theorems \ref{thmA}, \ref{thmC} and \ref{thmE}}
\label{sec:4}
In order to prove Theorems \ref{thmA} and \ref{thmE}, we consider a generalisation of the variation considered by Gibbons, Hartnoll, and Pope in \cite{GHP}. Using the terminology of \cite{GHP}, the perturbation is a `balloning mode' which generalises the manner one destabilises an ordinary Riemannian product by changing the relative volumes of the base and fibre. 
\begin{definition}[GHP variations]\label{GHPvardef}
	Let $(B\times F^{m}, \bar{g}\oplus f^{2}\tilde{g})$ be a warped product manifold. The GHP variations are the tensors
		\begin{equation}\label{dGHP}
		h=\left(\frac{f^{k}}{n}\bar{g}\right)\oplus\left(\frac{(m+k)f^{k+2}}{mn}\tilde{g}\right),
		\end{equation}
		where $k\in \mathbb{R}\backslash\{0\} $.
\end{definition}

\begin{proposition}\label{GHPprop}
	Let $(B^{n}\times F^{m},\bar{g}\oplus f^{2}\tilde{g})$ be a warped product Einstein manifold. Then any GHP variation $h$ is divergence-free. Furthermore if $c$ is the constant defined in Lemma \ref{Lem1}, then
	$$\langle N(h+cg),h+cg\rangle_{L^{2}(g)} =  C^1_{n,m,k}  \int_{B}f^{2k+m-2}|\overline{\nabla}f|^{2}dV_{\bar{g}}+\lambda\left( \|h\|^{2}-\frac{\left(\int_{M}\tr(h) dV_{g} \right)^{2}}{\|g\|^{2}}\right), $$
	where $$ C^1_{n,m,k} = -\frac{\vol(F)}{2}\left(\frac{k^{2}(4k + 2m + mn + (m+k)^{2})}{n^{2}m} \right).$$
\end{proposition}
\begin{proof}
Consider a variation
$$h=A f^{k}\bar{g}\oplus B f^{k+2}\tilde{g}.$$
A routine calculation using Lemma \ref{WPdiv} yields
$$\div(h) = \left((m+k)A-mB \right)f^{k-1}df.$$
Hence if the constants $A$ and $B$ are chosen so that
$$A=\frac{m}{m+k}B,$$
then $h$ is divergence-free. The GHP variation (\ref{dGHP}) has $A=1/n$ and $B=(m+k)/mn$ and so it is divergence-free.
Using Lemma \ref{LapLem} with 
$$\bar{h}=\frac{f^{k}}{n}\bar{g}, \qquad \tilde{h} = \tilde{g} \qquad \textrm{and} \qquad \psi = \frac{(m+k)f^{k+2}}{mn},$$
we obtain
$$\langle \Delta h,h\rangle = \frac{f^{k}\overline{\Delta}f^{k}}{n}+\frac{(m+k)^{2}}{mn^{2}}f^{k-2}\overline{\Delta}f^{k+2}-2\frac{(m+k)^{2}}{mn^{2}}f^{2k}\overline{\Delta}\log f+C^2_{n,m,k}f^{2k-2}|\overline{\nabla} f|^{2},$$
where
$$C^2_{n,m,k}=\frac{km^{3} - 4km^{2} - 8k^{2}m - 8km + k^{3}m - 6k^{2} - 4k^{3} - 4m^{2} + 2k^{2}m^{2} + km^{2}n}{mn^{2}}.$$
Integrating by parts yields
$$\int_{M}\langle \Delta h,h\rangle dV_{g}  = -\vol(F)\left(\frac{k^{2}(k^{2} + 2km + m^{2} + nm + 2)}{mn^{2}}\right)\int_{B}f^{2k+m-2}|\overline{\nabla}f|^{2}dV_{\bar{g}}.$$
Using Lemma \ref{COLem} (and simplifying using  Equations (\ref{KKeq}) and (\ref{QEMeq})) we obtain
$$\langle \Riem(h,\cdot),h\rangle = \left(\frac{(m+k)^{2}}{mn^{2}}-\frac{m+2k}{n^{2}}\right)f^{2k-1}\overline{\Delta}f+\lambda f^{2k}\left(\frac{1}{n}+\frac{(m+k)^{2}}{mn^{2}} \right).$$
Integrating by parts yields
$$\int_{M}\langle \Riem(h,\cdot),h\rangle dV_{g} = \vol(F)\left( \frac{-k^{2}(2k+m-1)}{mn^{2}}\right)\int_{B}f^{2k+m-2}|\overline{\nabla}f|^{2}dV_{\bar{g}}+\lambda \|h\|^{2}_{L^{2}(g)}.$$
The result follows immediately using the simplication given in Lemma \ref{Lem1}.
\end{proof}
The proofs of Theorems \ref{thmA} and \ref{thmE} now follow by taking special values of the parameter $k$.
\begin{proof}[Proof of Theorem \ref{thmA}]
As $k\neq 0$, no GHP variation is a multiple of the metric $g$ and the Cauchy-Schwarz inequality applied to $\langle h,g\rangle_{L^{2}(g)}$ implies that
$$
\|h\|^{2}-\frac{\left(\int_{M}\tr(h) dV_{g} \right)^{2}}{\|g\|^{2}}> 0.
$$
Hence we wish to select $k$ such that the term $C^{1}_{n,m,k}$ in Proposition \ref{GHPprop} is non-negative. Clearly this can only occur if 
$$(4k + 2m + mn + (m+k)^{2})\leq 0.$$
Viewed as a quadratic in $k$, $(4k + 2m + mn + (m+k)^{2})$ is minimised when 
$k=-(2+m)$ giving a value of $((n-2)m-4)$. Hence the coefficient  $C^{1}_{n,m,k}$ is only non-negative when $(n-2)m\leq 4$. The result follows noting that either $n=3$ and $m\in\{2,3,4\}$ or $n=4$ and $m=2$. 	This covers all possible six-dimensional products as if the base has dimension 2 then, by the rigidity theorem of Case--Shu--Wei \cite{CSW}, the function $f$ is constant and the product is trivially unstable. The fibre of a  warped product Einstein manifold with positive Einstein constant must be at least two-dimensional, as, by Myers's theorem the product cannot have infinite fundamental group.   	
\end{proof}	 
 
\begin{proof}[Proof of Theorem \ref{thmE}]
We take $k=-(m+n)$ in Proposition \ref{GHPprop}. In this case the GHP variation $h$ is divergence-free and trace-free. This means that the constant $c=0$ and so 
$$\langle N(h),h\rangle_{L^{2}(g)} = \frac{1}{2}\langle\Delta_{L}h+2\lambda h,h\rangle_{L^{2}(g)}.$$
Computing using Proposition \ref{GHPprop} and rearranging yields
$$\langle \Delta_{L}h,h\rangle_{L^{2}(g)} =\vol(F)\left(\frac{(m+n)^{2}(4n+2m-n^{2}-mn)}{n^{2}m} \right)\int_{B}f^{-2n-m-2}|\overline{\nabla}f|^{2}dV_{\bar{g}}.$$
This is manifestly non-negative when $n=3$ and $m=2$ or $m=3$. Hence the Lichnerowicz Laplacian, when restricted to divergence-free, trace-free tensors, has a non-negative eigenvalue $-\kappa$ for some $\kappa\leq 0$. The instability follows from Proposition \ref{BHstab} as
$$\kappa \leq 0 < \frac{(9-(3+m))\lambda}{4}$$
when $m=2$ or $m=3$.
\end{proof}
We use a similar method to prove Theorem \ref{thmC}.
\begin{proof}[Proof of Theorem \ref{thmC}]
We consider the perturbation
$$h = 0\oplus \sigma.$$
Decomposing the variation $h$ in the manner of (\ref{wptens}) we see that 
$$\bar{h}=0, \quad \tilde{h}=\sigma \textrm{ and } \psi = 1.$$
This immediately yields $\div(h)=0$ by Lemma \ref{WPdiv} and it is clear $h$ is trace-free as $\sigma$ is trace-free. Using  Lemma \ref{LapLem} and Lemma \ref{COLem} we obtain
$$\Delta h = f^{-2}\widetilde{\Delta}\sigma +(2(1-m)|\overline{\nabla} \log f|^{2}-2\overline{\Delta} \log f)\sigma,$$ 
and
$$\Riem(h,\cdot) =  f^{-2}\widetilde{\Riem}(\sigma,\cdot)+|\overline{\nabla}\log f|^{2}\sigma.$$
Hence
$$\langle\Delta_{L} h,h\rangle_{g} = f^{-6}\langle \widetilde{\Delta}_{L}\sigma,\sigma\rangle_{\tilde{g}}+2(\mu f^{-6}-\lambda f^{-4}+(2-m)f^{-6}|\overline{\nabla} f|^{2}-f^{-4}\overline{\Delta}\log f)|\sigma|_{\tilde{g}}^{2}.$$
Integrating by parts and using the fact ${\widetilde{\Delta}_{L}\sigma=-\kappa\sigma}$ we obtain
\begin{equation}\label{thmCeq}
\int_{M}\langle\Delta_{L} h,h\rangle_{g} dV_{g} =
2\|\sigma\|^{2}_{\tilde{g}}\int_{B} \left(\mu-\frac{\kappa}{2}\right) f^{m-6}-\lambda f^{m-4}-2f^{m-6}|\overline{\nabla}f|^{2}dV_{\bar{g}}.
\end{equation}
Multiplying Equation (\ref{KKeq}) by $f^{m-6}$ and integrating we see that
$$\int_{B}\frac{\mu}{2} f^{m-6}-2f^{m-6}|\overline{\nabla}f|^{2} dV_{\bar{g}} = \frac{\lambda}{2} \int_{B} f^{m-4}dV_{\bar{g}}.$$
Hence if $\kappa < \mu$, we can substitute into Equation (\ref{thmCeq}) and get the inequality
\begin{equation}\label{FUWPeq}
\int_{M}\langle \Delta_{L}h,h\rangle dV_{g} > -\lambda\|\sigma\|^{2}_{\tilde{g}}\int_{B}f^{m-4}dV_{\bar{g}} = -\lambda \|h\|^{2}_{g}.
\end{equation}
Hence the result follows.
\end{proof}
We can now prove Corollary \ref{Cor1}.
\begin{proof}[Proof of Corollary \ref{Cor1}]
 Einstein products and K\"ahler-Einstein metrics with $h^{1,1}>1$ (all with Einstein constant $\mu>0$) admit divergence-free, trace-free eigentensors of the Lichnerowicz Laplacian with eigenvalue $0$ (see \cite{CHI}, \cite{CZ}, and \cite{HMPAMS}). Equation (\ref{FUWPeq}) shows that a fibre-unstable warped product must also have a divergence-free, trace-free eigentensor satisfying the destabilising conditions and so must be unstable as a fibre. 
\end{proof} 

\section{The proof of Theorem \ref{thmB}}
\label{sec:5}
In \cite{HPW} He, Petersen, and Wylie introduced the following function $\rho$ and symmetric $(0,2)$-tensor $P$ associated to a quasi-Einstein metric $(B^{n},\bar{g},f,m)$ solving Equation (\ref{QEMeq}):
$$ \rho =  \frac{1}{m-1}\left((n-1)\lambda-\textrm{scal}(\bar{g})\right),$$ 
$$P = \Ric(\bar{g})-\rho \bar{g},$$
where $\textrm{scal}(\bar{g})$ is the scalar curvature of $\bar{g}$. Using the fact that
$$\Ric(\bar{g})  = mf^{-1}\overline{\nabla}^{2}f+\lambda \bar{g},$$
and
$$\textrm{scal}(\bar{g}) = mf^{-1}\overline{\Delta}f+n\lambda,$$
we can write the tensor $fP$ as
$$fP = m\overline{\nabla}^{2}f+\dfrac{m}{m-1}(\overline{\Delta}f+\lambda f)\bar{g}.$$
If the quasi-Einstein metrics appear as a family $(B^{n},\bar{g}_{i},f_{i},m_{i})$ with $\bar{g}_{i}\rightarrow \bar{g}_{\infty}$ and ${m_{i}\rightarrow \infty}$ as ${i\rightarrow \infty}$, then the tensors $P_{i}\rightarrow \Ric(\bar{g}_{\infty})$. As we saw in Proposition 2.6, the Ricci tensor of a gradient soliton plays an important role in the destabilising of product solitons. Hence we expect the tensor $P$ to play a similar role for warped product Einstein metrics. As we shall see, it is the tensor $fP$ that is divergence-free and so we have the following definition.

\begin{definition}[Ricci variation]\label{Ricvardef}
Let $(B^{n}\times F^{m}, \bar{g}\oplus f^{2}\tilde{g})$ be a warped product Einstein manifold with Einstein constant $\lambda$. Then the Ricci variation is the tensor
\begin{equation}\label{BHMvar}
h=\left(m\overline{\nabla}^{2}f +\frac{m}{m-1}(\overline{\Delta}f+\lambda f)\bar{g}\right) \oplus 0.
\end{equation}
\end{definition}
\begin{lemma} \label{Lem52}
The Ricci variation (\ref{BHMvar}) is divergence-free.
\end{lemma}
\begin{proof}
We begin by noting if $\bar{\eta} \in s^{2}(TB)$ satisfies $\overline{\div}(\bar{\eta})=0$, then using Lemma \ref{WPdiv},
$$\eta=f^{-m}\bar{\eta}\oplus 0,$$
satisfies $\div(\eta)=0$. In Proposition 5.4 of \cite{HPW}, the authors showed that $\overline{\div}(f^{m+1}P)=0$ and so $h=fP\oplus 0$ satisfies $\div(h)=0$.
\end{proof}
The hypothesis of Theorem \ref{thmB} will be used to guarantee that the functions $f_{i}$ have certain limiting behaviours.  We collect what we will need in a lemma. Note that  $(B,\bar{g}_{i},f_{i},m_{i})$ converging the $C^{\infty}$ topology to a non-trivial Ricci soliton $(B,\bar{g}_{\infty},\varphi)$ means $g_{i}\rightarrow g_{\infty}$, $f_{i}^{m_{i}}\rightarrow e^{-\varphi}$, and $m_{i}\rightarrow \infty$ as $i \rightarrow \infty$ where the convergence is uniform with respect to any $C^{k}$-norm.

\begin{lemma} \label{convLem}
	Let $(B,\bar{g}_{i},f_{i},m_{i})$ be a one parameter family of quasi-Einstein metrics that converge in the $C^{\infty}$ topology to a non-trivial Ricci soliton $(B,\bar{g}_{\infty},\varphi)$ as in Theorem \ref{thmB}. Then
	\begin{enumerate}
		\item The functions $f_{i} \rightarrow 1 $ as $i\rightarrow \infty$.
		\item The one-forms $m_{i}^{1/2} d\log f_{i}\rightarrow 0$ as $i\rightarrow \infty$. 
		\item The associated constants $\mu_{i}$ from Equation (\ref{KKeq}) satisfy $\mu_{i}\rightarrow \lambda$ as $i\rightarrow \infty$.
		
	\end{enumerate}
	
\end{lemma} 
\begin{proof}
The first two items follow trivially from the requirement that $f_{i}^{m_{i}}\rightarrow e^{-\varphi}$.  Item (3) is proved by Case in \cite{Case2} Proposition 4.11.
\end{proof}
We remark that the normalisation of the limiting soliton potential $\varphi$ is actually fixed by $\lim_{i\rightarrow \infty}m_{i}(\lambda-\mu_{i})$. This is also proved by Case in Proposition 4.11 of \cite{Case2}. 
Next, we compute the stability operator $N$ applied to the Ricci variation.
\begin{lemma}\label{rvarLem}
Let $(B^{n}\times F^{m}, \bar{g}\oplus f^{2}\tilde{g})$ be a warped product Einstein manifold with $\lambda$ and $\mu$ as in equation (\ref{KKeq}) and let $h$ be the Ricci variation (\ref{BHMvar}). Then
\begin{align*}
\langle \frac{1}{2}\Delta h+\Riem(h,\cdot),h\rangle_{L^{2}(g)} =\vol(F)\Bigg( &\lambda\|h\|_{L^{2}(f^{m}dV_{\bar{g}})}^{2}+\frac{m}{m-1}\Bigg\langle \frac{1}{2}\Delta(\overline{\Delta }f +\lambda f)\bar{g}\\
&+m\mu(f^{-3}df\otimes df),h\Bigg\rangle_{L^{2}(f^{m}dV_{\bar{g}})} \Bigg).\\
\end{align*}
\end{lemma}
\begin{proof}

For an arbitrary function $\Phi $ we recall that in coordinates
$$\nabla^{2}\Phi_{AB} = \frac{\partial^{2}\Phi}{\partial x_{A}\partial x_{B}}-\Gamma_{AB}^{C}\frac{\partial \Phi}{\partial x_{C}}.$$
Hence we can use the formula
$$\Gamma_{\alpha \beta}^{c} = -\tilde{g}_{\alpha \beta}f\overline{g}^{cd}(\overline{\nabla}_{d}f),$$
from Lemma \ref{WPcon} to obtain
$$\nabla^{2}\Phi=\overline{\nabla}^{2}\Phi\oplus f\langle \overline{\nabla}\Phi,\overline{\nabla} f\rangle \tilde{g}.$$
Using this, and the fact that ${g=\bar{g}\oplus f^{2}\tilde{g}}$, we can write the tensor $h$ in the following manner,
$$h = m\nabla^{2}f+\frac{m}{m-1}(\overline{\Delta}f+\lambda f)g - \frac{m}{m-1}(f^{2}\overline{\Delta}f+\lambda f^{3} +(m-1)f|\overline{\nabla}f|^{2})\tilde{g}.$$
Thus equation (\ref{KKeq}) yields ${f^{2}\overline{\Delta}f+\lambda f^{3} +(m-1)f|\overline{\nabla}f|^{2} = f\mu}$ and so
$$h = m\nabla^{2}f+\frac{m}{m-1}(\overline{\Delta}f+\lambda f)g-\frac{\mu m f}{m-1}\tilde{g}.$$	
We will now compute the integral ${\langle \frac{1}{2}\Delta h+\Riem(h,\cdot),h\rangle_{L^{2}(g)}}$ term-by-term.  Let
$$\mathcal{T}_{1} = m\nabla^{2}f, \quad \mathcal{T}_{2} = \frac{m}{m-1}(\overline{\Delta}f+\lambda f)g \quad \textrm{ and } \quad \mathcal{T}_{3} = \frac{\mu m f}{m-1}\tilde{g}.$$

 We begin by computing $\langle N(\mathcal{T}_{1}),h\rangle_{L^{2}(g)}$. For an arbitrary function $\Phi $, we note the identity
$$\langle\nabla^{2}\Phi,H\rangle_{L^{2}(g)} = \langle \div^{\ast}(\nabla\Phi),H\rangle_{L^{2}(g)} = \langle \nabla\Phi,\div(H) \rangle_{L^{2}(g)},$$
for any $H \in TM^{\ast}\otimes TM^{\ast}$.
Hence, as it is divergence-free by Lemma \ref{Lem52}, $h$ is $L_{2}-$orthogonal to the Hessian of any function.
In the proof of their Lemma 3.5 in \cite{CH}, Cao--He show that for an Einstein metric
$$\frac{1}{2}\Delta\nabla^{2}\Phi+\Riem(\nabla^{2}\Phi,\cdot)	=\nabla^{2}\Delta \Phi+\lambda\nabla^{2}\Phi,$$
for any function $\Phi$. Applying these two identities to $\mathcal{T}_{1}$ we obtain
$$\langle \frac{1}{2}\Delta (m\nabla^{2}f)+\Riem(m\nabla^{2}f,\cdot),h\rangle_{L^{2}(g)} = \langle m\nabla^{2}(\Delta f+\lambda f),h  \rangle_{L^{2}(g)} = 0 = \lambda \langle  m\nabla^{2}f,h\rangle_{L^{2}(g)}.
$$
Thus
$$\langle \frac{1}{2}\Delta \mathcal{T}_{1}+\Riem(\mathcal{T}_{1},\cdot),h\rangle_{L^{2}(g)} = \lambda \langle\mathcal{T}_{1},h\rangle_{L^{2}(g)}. $$
\\
To compute the integral for the second term $\mathcal{T}_{2}$ we note that
$$\Delta(\Phi g) = (\Delta \Phi)g,$$
for any function $\Phi$ and any Riemannian metric $g$. We also have
$$\Riem(\Phi g,\cdot) = \lambda \Phi g,$$
for any function $\Phi$ any Einstein metric $g$ with Einstein constant $\lambda$.  
Hence we conclude
$$\langle \frac{1}{2}\Delta \mathcal{T}_{2}+\Riem(\mathcal{T}_{2},\cdot),h\rangle_{L^{2}(g)} = \lambda \langle\mathcal{T}_{2},h\rangle_{L^{2}(g)}+\frac{m}{m-1}\left(\frac{1}{2}\Delta(\overline{\Delta}f+\lambda f)\right)g.$$
To deal with the final term $\mathcal{T}_{3}$ we note that pointwise $\langle \tilde{g},h\rangle = 0 = \langle \mathcal{T}_{3},h\rangle$. This is clear using the original definition of $h$ in Equation (\ref{BHMvar}). Using Lemma \ref{LapLem} and Lemma \ref{COLem} we obtain
$$ \langle \frac{1}{2}\Delta(f\tilde{g})+\Riem(f\tilde{g},\cdot), h\rangle_{L^{2}(g)} = \langle m(f^{-3}df\otimes df-f^{-2}\overline{\nabla}^{2}f),h\rangle_{L^{2}(g)}. $$
Hence
$$\langle \frac{1}{2}\Delta \mathcal{T}_{3}+\Riem(\mathcal{T}_{3},\cdot),h\rangle_{L^{2}(g)} = \lambda \langle\mathcal{T}_{3},h\rangle_{L^{2}(g)}+\frac{m \mu}{m-1}\left(\langle m(f^{-3}df\otimes df-f^{-2}\overline{\nabla}^{2}f),h\rangle_{L^{2}(g)} \right).$$
Putting the three calculations for the $\mathcal{T}_{i}$ together we see that $h$ 
is almost a $\lambda$-eigentensor for the stability operator $N$.
\begin{align*}\label{Errorterm}
\langle \frac{1}{2}\Delta h+\Riem(h,\cdot),h\rangle_{L^{2}(g)} =&  \lambda\|h\|_{L^{2}(g)}^{2}+\frac{m}{m-1}\Bigg\langle \frac{1}{2}\Delta(\overline{\Delta}f+\lambda f)\bar{g}
\end{align*}
\begin{align}
&-m\mu(f^{-3}df\otimes df-f^{-2}\overline{\nabla}^{2}f),h\Bigg\rangle_{L^{2}(g)}.
\end{align}

We now observe two further identities for an arbitrary function $\Phi \in C^{\infty}(B)$, firstly the general identity 
$$\nabla^{2}\Phi^{-1} = -\frac{1}{\Phi^{2}}\nabla^{2}\Phi+\frac{2}{\Phi^{3}}d\Phi\otimes d\Phi.$$
The second is that, as $h$ is both divergence-free and pointwise orthogonal to tensors of the form $0\oplus \tilde{h}$, we have
$$0=\langle \nabla^{2}\Phi,h\rangle_{L^{2}(g)} =\langle\overline{\nabla}^{2}\Phi\oplus f\langle \overline{\nabla}\Phi,\overline{\nabla} f\rangle \tilde{g},h\rangle_{L^{2}(g)} =\langle\overline{\nabla}^{2}\Phi,h\rangle_{L^{2}(g)}.$$
Hence equation (\ref{Errorterm}) becomes
$$\langle \frac{1}{2}\Delta h+\Riem(h,\cdot),h\rangle_{L^{2}(g)} = \lambda\|h\|^{2}_{L^{2}(g)}+\frac{m}{m-1}\langle(\frac{1}{2}\Delta(\overline{\Delta }f+\lambda f)\bar{g}+m\mu(f^{-3}df\otimes df),h\rangle_{L^{2}(g)}.$$
The result follows by computing the integrals over the base $B$ and fibre $F$. 
\end{proof}
We can now prove Theorem \ref{thmB}.
\begin{proof} [Proof of Theorem \ref{thmB}]
Using Lemma (\ref{Lem1}), we see that the warped product ${(B^{n}\times F^{m_{i}},\bar{g}_{i}\oplus f_{i}^{2}\tilde{g})}$ is unstable if
$$ \langle \frac{1}{2}\Delta h_{i}+\Riem(h_{i},\cdot),h_{i}\rangle_{L^{2}({g}_{i})}  -\lambda\frac{\left(\int_{M}\tr(h_{i}) dV_{{g}_{i}} \right)^{2}}{(n+m_{i})\vol(M)}>0,$$
where the $h_{i}$ are the Ricci variations defined by the metrics $\bar{g}_{i}$, the functions $f_{i}$ and the constants $m_{i}$.  Thus we consider the behaviour, as $i\rightarrow \infty$ of the sequence 
$$a_{i} = \vol(F_{i})^{-1}\left(\langle \frac{1}{2}\Delta h_{i}+\Riem(h_{i},\cdot),h_{i}\rangle_{L^{2}({g}_{i})}  -\lambda\frac{\left(\int_{M}\tr(h_{i}) dV_{{g}_{i}} \right)^{2}}{(n+m_{i})\vol(M)}\right).$$
The sequence $a_{i}$ does not in fact depend at all on the fibres $F_{i}$.\\ 
\\
Using Lemma \ref{rvarLem} we need to consider the limit as $i\rightarrow \infty$ of
$$
 \lambda\|h_{i}\|_{L^{2}(\bar{g}_{i})}^{2}+\frac{m_{i}}{m_{i}-1}\langle(\frac{1}{2}\Delta(\overline{\Delta}f_{i}+\lambda f_{i})\bar{g_{i}}+m\mu(f_{i}^{-3}df_{i}\otimes df_{i}),h_{i}\rangle_{L^{2}(\bar{g}_{i})}.
$$
\\
As Lemma \ref{convLem} gives $f_{i} \rightarrow 1$, $d\log f_{i}^{m_{i}^{1/2}} \rightarrow 0$, and $\mu_{i}\rightarrow \lambda$ we see
$$\langle(\frac{1}{2}\Delta(\overline{\Delta }f_{i}+\lambda f_{i})\bar{g}_{i}+m_{i}\mu_{i}(f_{i}^{-3}df_{i}\otimes df_{i}),h_{i}\rangle_{L^{2}(f_{i}^{m_{i}}dV_{\bar{g}_{i}})} \rightarrow 0.$$
We also have
$$\vol(F_{i})^{-1} \lambda\frac{\left(\int_{M}\tr(h_{i}) dV_{g_{i}} \right)^{2}}{(m_{i}+n)\vol(M)} = \lambda\frac{\left(\int_{B}\tr(h_{i}) f_{i}^{m_{i}}dV_{\bar{g}_{i}} \right)^{2}}{(m_{i}+n)\vol(B,f^{m_{i}}dV_{\bar{g}_{i}})} \rightarrow 0.$$

The limit of the $a_{i}$ is thus
$$\lambda\|h_{\infty}\|^{2}_{L^{2}(e^{-\varphi}dV{\bar{g}_{\infty}})},$$
where $h_{\infty} =\Ric(\bar{g}_{\infty})$ is the limit of the $h_{i}$ as $i \rightarrow \infty$ (the limit exists as $g_{i}$  converges in the $C^{\infty}$ topology). 
 Hence we see for large enough $i$,
$$\int_{M}\langle N(h_{i}+cg_{i}),h_{i}+cg_{i}\rangle \ d V_{g} > 0,$$
and the result follows.
\end{proof}	
\bibliography{QEMRefs}

\begin{thebibliography}{10}

\bibitem{BHJM}
{\sc Batat, W., Hall, S.~J., Jizany, A., and Murphy, T.}
\newblock Conformally {K}\"ahler geometry and quasi-{E}instein metrics.
\newblock {\em M\"unster Journal of Mathematics 8\/} (2015), 211--228.

\bibitem{Bes}
{\sc Besse, A.~L.}
\newblock {\em Einstein manifolds}.
\newblock Classics in Mathematics. Springer-Verlag, Berlin, 2008.
\newblock Reprint of the 1987 edition.

\bibitem{Bo}
{\sc B{\"o}hm, C.}
\newblock Inhomogeneous {E}instein metrics on low-dimensional spheres and other
  low-dimensional spaces.
\newblock {\em Invent. Math. 134}, 1 (1998), 145--176.

\bibitem{Cao}
{\sc Cao, H.-D.}
\newblock Existence of gradient {K}\"ahler-{R}icci solitons.
\newblock In {\em Elliptic and parabolic methods in geometry ({M}inneapolis,
  {MN}, 1994)}. A K Peters, Wellesley, MA, 1996, pp.~1--16.

\bibitem{CaoRPRS}
{\sc Cao, H.-D.}
\newblock Recent progress on {R}icci solitons.
\newblock In {\em Recent advances in geometric analysis}, vol.~11 of {\em Adv.
  Lect. Math. (ALM)}. Int. Press, Somerville, MA, 2010, pp.~1--38.

\bibitem{CHI}
{\sc Cao, H.-D., Hamilton, R., and Ilmanen, T.}
\newblock Gaussian densities and stability for some {R}icci solitons.
\newblock {\em -\/} (2004).
\newblock preprint, arXiv:math/0404165 [math.DG].

\bibitem{CH}
{\sc Cao, H.-D., and He, C.}
\newblock Linear stability of {P}erelman's {$\nu$}-entropy on symmetric spaces
  of compact type.
\newblock {\em J. Reine Angew. Math. 709\/} (2015), 229--246.

\bibitem{CZ}
{\sc Cao, H.-D., and Zhu, M.}
\newblock On second variation of {P}erelman's {R}icci shrinker entropy.
\newblock {\em Math. Ann. 353}, 3 (2012), 747--763.

\bibitem{CSW}
{\sc Case, J., Shu, Y.-J., and Wei, G.}
\newblock Rigidity of quasi-{E}instein metrics.
\newblock {\em Differential Geom. Appl. 29}, 1 (2011), 93--100.

\bibitem{Case2}
{\sc Case, J.~S.}
\newblock Smooth metric measure spaces and quasi-{E}instein metrics.
\newblock {\em Internat. J. Math. 23}, 10 (2012), 1250110, 36.

\bibitem{Case}
{\sc Case, J.~S.}
\newblock The energy of a smooth metric measure space and applications.
\newblock {\em J. Geom. Anal. 25}, 1 (2015), 616--667.

\bibitem{DWW}
{\sc Dai, X., Wang, X., and Wei, G.}
\newblock On the stability of {R}iemannian manifold with parallel spinors.
\newblock {\em Invent. Math. 161}, 1 (2005), 151--176.

\bibitem{DWC1}
{\sc Dancer, A., and Wang, M.}
\newblock On {R}icci solitons of cohomgeneity one.
\newblock {\em Ann. Global Anal. Geom. 39\/} (2011), 259--292.

\bibitem{GH}
{\sc Gibbons, G.~W., and Hartnoll, S.~A.}
\newblock Gravitational instability in higher dimensions.
\newblock {\em Phys. Rev. D (3) 66}, 6 (2002), 064024, 17.

\bibitem{GHP}
{\sc Gibbons, G.~W., Hartnoll, S.~A., and Pope, C.~N.}
\newblock Bohm and {E}instein-{S}asaki metrics, black holes, and cosmological
  event horizons.
\newblock {\em Phys. Rev. D (3) 67}, 8 (2003), 084024, 24.

\bibitem{HHS}
{\sc Hall, S., Haslhofer, R., and Siepmann, M.}
\newblock The stability inequality for {R}icci-flat cones.
\newblock {\em J. Geom. Anal. 24}, 1 (2014), 472--494.

\bibitem{Halljgp}
{\sc Hall, S.~J.}
\newblock Quasi-{E}instein metrics on hypersurface families.
\newblock {\em J. Geom. Phys. 64\/} (2013), 83--90.

\bibitem{HMPAMS}
{\sc Hall, S.~J., and Murphy, T.}
\newblock On the linear stability of {K}\"ahler-{R}icci solitons.
\newblock {\em Proc. Amer. Math. Soc. 139}, 9 (2011), 3327--3337.

\bibitem{HMAGAG1}
{\sc Hall, S.~J., and Murphy, T.}
\newblock On the spectrum of the {P}age and the {C}hen-{L}e{B}run-{W}eber
  metrics.
\newblock {\em Ann. Global Anal. Geom. 46}, 1 (2014), 87--101.

\bibitem{HMPEMS}
{\sc Hall, S.~J., and Murphy, T.}
\newblock Numerical approximations to extremal toric {K}\"ahler metrics with
  arbitrary {K}\"ahler class.
\newblock {\em Proc. Edinb. Math. Soc. (2) 60}, 4 (2017), 893--910.

\bibitem{HPW}
{\sc He, C., Petersen, P., and Wylie, W.}
\newblock On the classification of warped product {E}instein metrics.
\newblock {\em Comm. Anal. Geom. 20}, 2 (2012), 271--311.

\bibitem{KK}
{\sc Kim, D.-S., and Kim, Y.~H.}
\newblock Compact {E}instein warped product spaces with nonpositive scalar
  curvature.
\newblock {\em Proc. Amer. Math. Soc. 131}, 8 (2003), 2573--2576.

\bibitem{Koi}
{\sc Koiso, N.}
\newblock On rotationally symmetric {H}amilton's equation for
  {K}\"ahler-{E}instein metrics.
\newblock In {\em Recent topics in differential and analytic geometry}, vol.~18
  of {\em Adv. Stud. Pure Math.} Academic Press, Boston, MA, 1990,
  pp.~327--337.

\bibitem{KlKr3}
{\sc Kr\"oncke, K.}
\newblock Stable and unstable {E}instein warped products.
\newblock {\em Trans. Amer. Math. Soc. 369}, 9 (2017), 6537--6563.

\bibitem{KlKr4}
{\sc Kr\"oncke, K.}
\newblock Stability of sin-cones and cosh-cylinders.
\newblock {\em Ann. Sc. Norm. Super. Pisa Cl. Sci. (5) 18}, 3 (2018),
  1155--1187.

\bibitem{KlKr1}
{\sc Kr\"oncke, K.}
\newblock Stability of {E}instein metrics under {R}icci flow.
\newblock {\em Comm. Anal. Geom.\/} (to appear).
\newblock arXiv:1312.2224 [math.DG].

\bibitem{LPP}
{\sc L{\"u}, H., Page, D.~N., and Pope, C.~N.}
\newblock New inhomogeneous {E}instein metrics on sphere bundles over
  {E}instein-{K}\"ahler manifolds.
\newblock {\em Phys. Lett. B 593}, 1-4 (2004), 218--226.

\bibitem{Pali2}
{\sc Pali, N.}
\newblock The soliton-{R}icci flow with variable volume forms.
\newblock {\em Complex Manifolds 3\/} (2016), Art. 3.

\bibitem{Pali1}
{\sc Pali, N.}
\newblock Variational stability of {K}\"ahler-{R}icci solitons.
\newblock {\em Adv. Math. 290\/} (2016), 15--35.

\bibitem{Per1}
{\sc Perelman, G.}
\newblock The entropy formula for the {R}icci flow and its geometric
  applications.
\newblock {\em -\/} (2002).
\newblock preprint, arXiv:math/0211159 [math.DG].

\bibitem{Per3}
{\sc Perelman, G.}
\newblock Finite extinction time for the solutions to the {R}icci flow on
  certain three-manifolds.
\newblock {\em -\/} (2003).
\newblock preprint, arXiv:math/0307245 [math.DG]].

\bibitem{Per2}
{\sc Perelman, G.}
\newblock {R}icci flow with surgery on three-manifolds.
\newblock {\em -\/} (2003).
\newblock preprint, arXiv:math/0303109 [math.DG].

\bibitem{Sesum}
{\sc Sesum, N.}
\newblock Linear and dynamical stability of {R}icci-flat metrics.
\newblock {\em Duke Mathematical Journal 133}, 1 (2006), 1--26.

\end{thebibliography}
\bibliographystyle{acm}

\end{document}